\documentclass{amsartvera}

\usepackage{mathrsfs, amsfonts, amsmath, amssymb, amsthm}
\usepackage[all]{xy}
\usepackage{pdfpages}
\usepackage{graphicx}
\usepackage{fancyhdr}
\usepackage{lipsum}
\usepackage{charter}
\usepackage{hyperref}
\usepackage{wasysym}
\usepackage{verbatim}
\newcommand{\R}{\mathbb{R}}
\newcommand{\C}{\mathbb{C}}
\newcommand{\Ctwo}{\mathbb{C}^{2}}
\newcommand{\map}{\longrightarrow}

\newtheorem{theorem}{Theorem}[section]
\newtheorem{lemma}[theorem]{Lemma}
\newtheorem{proposition}[theorem]{Proposition}

\newtheorem{sptheorem}{Theorem}

\newtheorem{spcorollary}{Corollary}
\theoremstyle{definition}
\newtheorem{definition}[theorem]{Definition}
\newtheorem{example}{Example}
\newtheorem*{rem}{Remark}

\title{Near-symplectic 2n--manifolds}
\author{Ram{\'o}n Vera}
\email{rvera.math@gmail.com, rxv15@psu.edu}
\address{Department of Mathematics, The Pennsylvania State University, State College, PA, 16802, USA
\\
Department of Mathematical Sciences, Durham University, Science Laboratories,  South Rd., Durham DH1 3LE, UK}

\begin{document}
\begin{abstract}
\noindent We give a generalization of the concept of near-symplectic structures to $2n$ dimensions. 
According to our definition, a closed 2-form on a $2n-$manifold $M$ is near-symplectic, if it is symplectic 
outside a submanifold $Z$ of codimension 3, where $\omega^{n-1}$ vanishes. 
We depict how this notion relates to near-symplectic 4-manifolds and broken Lefschetz fibrations via some examples. We define a generalized broken Lefschetz fibration, or BLF,  as a singular map with indefinite folds and Lefschetz-type singularities. We show that given such a map on a $2n$-manifold over a symplectic base of codimension 2, then the total space carries such a near-symplectic structure, whose singular locus corresponds precisely to the singularity set of the fibration. A second part studies the geometry around the codimension--3 singular locus $Z$ .  We describe a splitting property of the normal bundle $N_Z$ that is also present in dimension four.  A tubular neighbourhood for $Z$ is provided, which has as a corollary a Darboux-type theorem for near-
symplectic forms. 
\end{abstract}

\maketitle

\section{Introduction}
The motivation for near-symplectic manifolds arose from a program initiated by Taubes to study 4-manifolds equipped with symplectic forms that vanish on circles with the goal of obtaining smooth invariants of non-symplectic 4-manifolds \cite{Taubes}.  A 4-manifold is called near-symplectic if it is equipped with a closed 2-form
that is non-degenerate outside a disjoint union of circles, where
it vanishes.
These structures where studied in detail in the work of Auroux,
Donaldson and Katzarkov \cite{ADK} using broken Lefschetz fibrations
(BLFs). It was shown that there is a direct correspondence between
BLFs and near-symplectic 4-manifolds. These results extended the
theorems of Donaldson \cite{Donaldson} and Gompf \cite{Gompf} on
Lefschetz fibrations and symplectic manifolds, which in turn expanded
Thurston's theorem on symplectic fibrations. Broken
Lefschetz fibrations have found fruitful application in low-dimensional topology, for
example in holomorphic quilts \cite{Quilts} and Lagrangian matching
invariants \cite{Lagrangian, Lagrangian2}. A
relevant existence result states that every smooth closed oriented
4-manifold admits a BLF \cite{Gay, Lekili, Baykur2, BaykurTopology, Akbulut}.
The geometric structure induced by a near-symplectic 4-manifold on 
the boundary of the tubular neighbourhood of its singular locus is an overtwisted
structure as studied by Honda, Gay
and Kirby \cite{Honda, Gay2}. This shows that these manifolds are not fillable as that would require removing all singular circles, which Perutz proved not to be possible \cite{Perutz}. 

This work aims to find a good notion to generalize near-symplectic 
structures on higher dimensions. We propose a definition on manifolds of dimensions $2n$ and use singular maps that resemble broken Lefschetz fibrations. We also study the underlying geometric structure, induced by the near-symplectic 
form, on the boundaries of tubular neighbourhoods, which are 
codimension 1 submanifolds in this setting. 

In section \ref{definition forms}, we suggest a definition of a
near-symplectic structure in dimension $2n$. The goal  is to relax the
non-degeneracy condition of the symplectic form in a controlled way so that 
it degenerates exclusively on a certain submanifold. The idea starts by considering 
a closed 2-form $\omega$ on a smooth, orientable,
$2n$-manifold $M$, such that $\omega^n \geq 0$. At the points where the degeneracy 
occurs, that is where $\omega^n =0$, we impose a transversality condition on the gradient 
or differential map of $\omega$. This transversality condition tells us that the singular locus $Z$ is a submanifold 
of codimension 3 in $M$, where $\omega^{n-1}_{p}=0$ for all $p\in Z$. We call these 2-forms near-symplectic. Examples of near-symplectic $2n$-manifolds are given in \ref{examples}

Next, we study the question of the existence of these structures
using singular fibrations, analogous to BLFs. We define a generalized {\it BLF}
as a submersion $f\colon M^{2n} \rightarrow X^{2n-2}$ with two types of sets of singularities, both of which lie in $M$. First, we have
codimension 4 submanifolds of extended Lefschetz type singularities locally modelled by complex
coordinate charts $\C^n \rightarrow \C^{n-1}$ such that $(z_1, \dots, z_n) \mapsto (z_1,\dots, z_{n-2}, z_{n-1}^2 + z_{n}^2)$. The second singularities are codimension 3-submanifolds $\Sigma$ of indefinite folds modelled by real coordinate charts 
$\R^{2n} \rightarrow \R^{2n-2}, (t_1, \dots, t_{2n-3}, x_1, x_2,  x_3) \mapsto (t_1, \dots, t_{2n-3}, -x_1^2 + x_2^2 + x_3^2)$.

\begin{sptheorem}\label{theorem 1}
Let $f\colon M \rightarrow X$ be a generalized BLF from a 
smooth closed oriented $2n$-manifold $M$ to a compact symplectic $(2n-2)-$manifold $(X, \omega_{X})$. Denote by $\Sigma$ the set of fold singularities of $f$. Assume that there is a class $\alpha \in H^2(M)$, such that it pairs 
positively with every component of every fibre, and that $\alpha|_{\Sigma} = [\omega_X|_{\Sigma}]$. Then, there 
is a near-symplectic form $\omega$ on $M$, with singular locus $Z$, 
equal to $\Sigma$, and symplectic fibres outside $\Sigma$.
\end{sptheorem}

The proof appearing in section \ref{proof theorem 1} starts by constructing an explicit closed 2-form on the fibres that vanishes at the set of singularities of the mapping. Then it pulls back the symplectic form of the base. 
Both 2-forms are combined and glued together into a global 2-form representing the class  $\alpha$. This statement follows a similar line of reasoning as Auroux, Donaldson, and Katzarkov's  \cite{ADK} construction of near-symplectic forms using BLFs in dimension 4.

The last section concerns the geometric structure on the boundary of the neighbourhood of the singular locus. We study two geometric structures that appear on a codimension 1 submanifold of $M$. Firstly, we look at Hamiltonian structures. A Hamiltonian structure on an $(2n-1)$-dimensional manifold $N$ is a closed 2-form $\Omega$ such that $\Omega^{n-1}\not=0$ everywhere.  In the presence of a Hamiltonian structure, there is a 1-dimensional distribution associated to $\Omega$ through its kernel $\ker(\Omega)$.  A 1-form $\lambda$ is called a stabilizing 1-form, if $\lambda \wedge \Omega^{n-1}>0$ and $\ker(\Omega) \subset \ker(d\lambda)$. The pair $(\lambda, \Omega)$ is known as a stable Hamiltonian structure. 
A near-symplectic form naturally equips the singular locus $Z$ with a Hamiltonian structure.  Moreover, if $Z$ carries a stable Hamiltonian structure so does the boundary of a small tubular neighbourhood in case that the normal bundle is trivial.  

We conclude by examining the  properties of the normal bundle of $Z$ that are defined by the near-symplectic form.  As in dimension 4, there is a decomposition of the normal bundle $N_Z$ in two subbundles, a rank 1 bundle $L^{-}$ and a rank 2 bundle $L^{+}$.  In section \ref{neighbourhood thm}, we give a neigbourhood theorem for near-symplectic forms around their singular locus. 
\begin{sptheorem}\label{theorem 2}
Let $(M_0, \omega_0), (M_1, \omega_1)$ be two near-symplectic manifolds with diffeomorphic singular locus $Z_0 \cong Z_1$ and equal symplectic forms on them, $\omega_0|_{Z_0} = \omega_1|_{Z_1}$.  Assume that there is an isomorphism on the normal bundles  $N_{Z_{0}} \simeq N_{Z_{1}}$, such that it restricts to an isomorphism on the positive subbundles $L^{+}_{0} \simeq L^{+}_{1}$.  Denote by $\mathcal{U}_{0}\subset M_0$ and $\mathcal{U}_{1}\subset M_1$ the corresponding tubular neighbourhoods of $Z_0$ and $Z_1$.  Then, there is a homeomorphism $\varphi\colon \mathcal{U}_{0} \rightarrow \mathcal{U}_{1}$ that is a diffeomorphism away from $Z$, such that $\varphi^{*}\omega_1 = \omega_0$.
\end{sptheorem}
As a corollary, we obtain a local Darboux-type theorem which describes a near-symplectic form around a point of $Z$. 
\begin{spcorollary}
Let $(M, \omega)$ be a near-symplectic manifold and $p$ a point of the singular locus $Z\subset M$.  There is a coordinate neighbourhood $U\subset M$ around $p$, such that on $U$ 
$$ \omega = \omega_Z - 2 x_1 (dz_0 \wedge dx_1 + dx_2 \wedge dx_3 ) +   x_2 (dz_0 \wedge dx_2 - dx_1 \wedge dx_3 ) + x_3 (dz_0 \wedge dx_3 + dx_1 \wedge dx_2) $$
where $\omega_Z:= i^{*}\omega$ is a closed 2-form of maximal rank on $Z$. 
\end{spcorollary}

\section*{Acknowledgements}
I would like to express my thanks to my advisor, Dirk Sch{\"u}tz, for all his
support and valuable discussions. I am also very grateful to Denis Auroux and
Yanki Lekili for many interesting questions and remarks at various stages
of this work. Special thanks to Klaus Niederkr{\"u}ger for his comments on the geometry around the singular locus.  I also want to thank Tim Perutz for his email correspondence on the definition of near-symplectic forms. 

\noindent This work was supported by DDF09/FT/000092234 and by CONACYT 212591.

\section{Near-symplectic forms}
We first recall the definition of near-symplectic forms on 4-manifolds \cite{ADK}.

\begin{definition}
Let $X$ be a smooth oriented 4-manifold. Consider a closed 2-form $\omega \in \Omega^2(X)$ such that $\omega^2\geq 0$ and such that $\omega_p$ only has rank 4 or rank 0 at any point $p \in X$, but never rank 2. The form $\omega$ is called {\it near-symplectic}, if it is non-degenerate or it vanishes transversally along circles.  That is, for every $p\in X$, either

\begin{enumerate}
 \item $\omega_p^2>0$, or
 \item $\omega_p = 0$, and $\textnormal{Rank}(\nabla \omega_p) = 3$, where $\nabla\omega_p\colon T_pX \rightarrow \Lambda^2 T^{*}_{p} X$ denotes the intrinsic gradient of $\omega$.
\end{enumerate}
\end{definition}

It follows from the condition on $\nabla \omega_p$ that the singular locus $Z_{\omega}$ is a smooth 1-submanifold of $X$ \cite{ADK}, \cite{Perutz}.  A prototypical example of a near-symplectic 4-manifold is given by $X = S^1\times Y^3$, where $Y$ is a closed 3-manifold.  Consider a closed 1-form $\alpha\in\Omega^1(Y)$ with indefinite Morse critical points and let $t$ be the parameter of $S^1$.  The 2-form $\omega = dt\wedge \alpha + \ast(dt\wedge \alpha)$ is near-symplectic, where the Hodge $\ast$-operator is defined with respect to the product metric on $S^1$ and $Y$.  The singular locus $Z_{\omega} = \left\{ p \in X \mid \omega_p=0 \right\}$ is in this case  $S^1\times \textnormal{Crit}(\alpha)$.

\subsection{Near-symplectic 2n-manifolds}\label{definition forms}
The following definition of near-symplectic forms in higher dimensions is  due to Tim Perutz.  The author would also like to acknowledge that the coming exposition follows very closely a message from Perutz. 

Let $M$ be an oriented smooth $2n$-manifold, and $\omega\in \Omega^{2}(M)$ a closed 2-form such that
\begin{equation}\label{assumption}
\omega^n \geq 0
\end{equation}
everywhere.  Suppose that at some point $p$, the kernel $K$ of $\omega$ seen as a subspace of the tangent space has dimension $4$.
$$K = \lbrace v\in T_pM \mid \omega_p (v, \cdot) = 0 \rbrace $$
We have an intrinsic gradient $\nabla \omega \colon K \rightarrow \Lambda^2 T_{p}^{*} M$.  We can restrict this map to bivectors in $K$ to get a map $K\rightarrow \Lambda^2 T_p^{*}M \rightarrow \Lambda^2 K^{*}$, where the map $\Lambda^2 T_p^{*}M \rightarrow \Lambda^2 K^{*}$ corresponds to the dual of the inclusion $K\hookrightarrow T_p M$ in the corresponding exterior algebra. We denote this composition as
\begin{equation}\label{DK}
D_{K} \colon K \rightarrow \Lambda^{2} K^{*}
\end{equation}
Then the wedge square gives us a non-degenerate quadratic form $q\colon \Lambda^2 K^{*}\otimes \Lambda^2 K^{*} \rightarrow \Lambda^{4} K^{*}$.

\begin{proposition}
The image $\textnormal{Im} (D_K)$ is dimension at most 3.  In local coordinates this is a positive semi-definite subspace of $\Lambda^2 K^{*}$ with respect to the wedge square form. That is, the 4-form $D_K(v) \wedge D_K(v) \geq 0$ for $v\in K$.
\end{proposition}
\begin{proof}
Take an arbitrary tangent vector $v\in T_p M$ and choose coordinates such that $p=0$ is the point at the origin. By our assumption on $\omega$, we have $\omega^n (t\cdot v) \geq 0$ for all scalars $t$, where $t\cdot v$ points into the manifold.   Yet, if we use a Taylor expansion to write  $\omega (t\cdot v) = \omega(0) + t\cdot \nabla_v \omega(0) + O(t^2)$   and take $v\in K$, we have
$$\omega^n (t\cdot v) = \underbrace{\omega^n(0)}_{=0} + t {n\choose 1} \underbrace{\omega^{n-1}(0)}_{=0} \wedge \nabla_v \omega (0) + t^2  {n\choose 2}  \omega(0)^{n-2} \wedge \left(\nabla \omega_v (0)\right)^2 + O(t^3)$$
The forms $\omega^n(0)$ and $\omega^{n-1}(0)$ vanish since they necessarily take vectors $\partial_{k_{1}}, \dots , \partial_{k_{4}}$ from $K$, whereas in the linear combination of $\omega^{n-2}$ there will be vectors outside from $K$ where the form remains non-zero. This gives us
$$\omega^n (t\cdot v) = {n\choose 2} \cdot t^2 \cdot \omega(0)^{n-2} \wedge \left(\nabla \omega_v (0)\right)^2 + O(t^3)$$
We work in local coordinates using the tangent space at $p=0$ for the local coordinate system.  The space $T_pM \slash K$ has a symplectic structure and we can combine an orientation on it with an orientation of $K$ to get an orientation of $T_pM$, which has a natural orientation. With respect to this chosen orientation we want to show that $D_K(v) \wedge D_K(v) \geq 0$ for a $v\in K$.  Let $e_i = \left(\frac{\partial}{\partial x_i}\right)_{1\leq i \leq 2n}$ be an oriented basis.  Since $\omega^n(tv) \geq 0$ from our original consideration \eqref{assumption}, then we have that $\omega^n (t\cdot v) (e_1, \dots , e_{2n}) \geq 0$, thus
$$\omega^n (t\cdot v) \approx C \cdot \omega(0)^{n-2} \wedge \left(\nabla \omega_v (0)\right)^2 (e_1, \dots, e_{2n}) \geq 0 $$
with the constant $C = {n\choose 2} \cdot t^2$. The form $\omega(0)^{n-2}$ has a sign on the complementary subspace to $K$ in $T_pM$, since we have chosen an orientation.  However, from \eqref{DK} by restricting to vectors in $K$, then
$$\omega^n (t\cdot v) \approx C\cdot t^2 \cdot \omega(0)^{n-2}(e_1, \dots, e_{2n-4})  \wedge \underbrace{\left(\nabla \omega_v (0)\right)^2}_{D_K(v)\wedge D_K(v)} (\partial_{k_{1}}, \dots , \partial_{k_{4}}) \geq 0  $$
We can see now that the image of $D_K$ is a positive semi-definite subspace of $\Lambda^2 K^{*}$.  Hence $\textnormal{Im} (D_K)$ has dimension at most 3.  In particular, $D_K(v) \wedge D_K (v)$ is a non-negative 4-form with respect to $K$.
\end{proof}

\begin{definition}\label{near-symplectic}
The 2-form $\omega \in \Omega^2 (M^{2n})$ is {\it near-symplectic}, if it is closed, $\omega^n \geq 0$, and at a point $p$ where $\omega^n = 0$, one has that the kernel $K$ is 4-dimensional and that the Im$(D_K)$ has dimension 3.
\end{definition}
\begin{rem}
Informally, the definition implies that a closed 2-form $\omega \in \Omega^2(M)$ is {\it near-symplectic}, if for every $p\in M$, either
\begin{itemize}
 \item[(i)] $\omega_p^n > 0$, or
 \item[(ii)] $\omega_p^{n-1} = 0$, but $\omega_p^{n-2} \not= 0$ at a codimension 3 submanifold of $M$.
\end{itemize}
In the remaining part of this section we will explain why the degeneracy locus is a codimension 3 submanifold.
\end{rem}
The image of the map $D_{K} \colon K \rightarrow \Lambda^{2} K^{*}$ is of dimension 3, thus it has 1-dimensional kernel. If we look at $\omega^{n-1}$ then it vanishes at $p$, since it takes at least 2 vectors from $K$. Moreover, $G= \nabla \omega^{n-1}(p)$ is 
%
%
%
defined.  Choose coordinates $(x_k)$ so that $K$ is defined by the vanishing of all but the last four $dx_k$. Take the derivative of $\omega^{n-1}$ and apply the chain rule to obtain
$$G= (n-1) \omega(p)^{n-2} \nabla \omega_p $$
where the gradient on the right is defined using the coordinates. The 2-form $\omega$ is symplectic on the submanifold $Z$ where the last 4 coordinates are zero.  We can adjust the coordinates to Darboux form, so that $\omega$ is constant on $Z$, that is, $\omega|_{p} = dx_1 \wedge dx_2 + \dots + dx_{2n-5}\wedge dx_{2n-4}$ for $p\in Z$. Hence $\nabla \omega_p (\partial x_{i}) = 0$ for $i = 1, \dots , 2n-4$. However,
$$ \ker G = \ker (\nabla \omega_p)$$
and now one sees that this is a codimension 3 subspace containing the line $\ker (D_K)$.  Hence the degeneracy locus $Z$ of the near-symplectic form is a codimension 3-submanifold of $M$.

\begin{lemma}
The singular locus $Z_{\omega} = \{ p\in M \mid \omega_{p}^{n-1} = 0 \}$ is a codimension 3 submanifold of $M$.
\end{lemma}
\begin{rem}
The property of $\omega|_{V \setminus Z}^n>0$ guarantees that the whole $V^{2n}$ is orientable. This is due to the fact that $Z$ is a submanifold of codimension 3.  In fact, it follows from a standard algebraic topological argument that this orientability property  is true on any dimension if the codimension of the submanifold is greater or equal to two.  That is to say, if $\omega$ is a 2-form on a $2n$-manifold $V$, $K$ is a $k$-dimensional submanifold of $V$, and $\omega^n >0$ on $V\setminus K$, then $V$ is oriented if $\textnormal{codim}(K) \geq 2$.
\end{rem}

\begin{rem}
In dimension 4, near-symplectic structure are related to self-dual harmonic forms.
An obvious obstacle in dimensions 6 and above is that there is no analogue of self-duality
for 2-forms. 
Some local models of near-symplectic forms on 6-manifolds $M^6$ seem
to indicate that near-symplectic forms could be equivalent to $\omega = \ast \omega^2$ for some generic metric, outside the singular locus $Z$.
\end{rem}

\subsection{Examples}\label{examples}
\begin{example}
On $\R^{2n}$ with coordinates $(q_1, p_1, \dots, q_{n-2}, p_{n-2}, x_0, x_1, x_2, x_3)$, the following 2-form is near-symplectic
\begin{align}
\omega =& - 2x_1 (dx_0 \wedge dx_1 + dx_2 \wedge dx_3) + x_2 (dx_0 \wedge dx_2 - dx_1 \wedge dx_3)
\nonumber
\\
& + x_3 (dx_0 \wedge dx_3 + dx_1 \wedge dx_2) + \sum_{i=1}^{n-2} dq_i \wedge dp_i
\nonumber
\end{align}
The singular locus where $\omega^{n-1} = 0$ is given by $Z_{\omega} = \left\{ p\in \R^{2n} \mid x_1 = x_2 = x_3 = 0 \right\}$ and $\omega^n > 0$ away from $Z_{\omega}$.
 \end{example}

For the next example, let $(Q, \bar{\omega})$ be a symplectic manifold of and $\phi\colon Q \rightarrow Q$ a symplectomorphism. Form a mapping torus $N = Q (\phi) = Q\times [0,1] \slash (x,0) \simeq (\phi(x), 1)$. The mapping torus is in particular a fibre bundle over $S^1$ and it carries a non-vanishing closed 1-form $\beta$. 
We can extend $\bar{\omega}$ from $Q$ to $N$.  There is a 2-form defined on $Q\times \R$.  The $\mathbb{Z}$-action on this manifold given by $(x,t) \mapsto (\phi(x), t+1)$ leaves the 2-form invariant, hence it descends to the quotient. Thus, $\bar{\omega}$ is a well-defined 2-form on $N$ that is symplectic on $Q$.

\begin{example}
Consider the $2n$-manifold $M= N\times Y$ obtained by crossing $N$ with a closed, connected, orientable, smooth 3-manifold $Y^3$.  Let $\alpha \in \Omega^1(Y)$ be a closed 1-form with indefinite Morse singular points (i.e. no maximum and no minimum). By Calabi's and Honda's theorems \cite{Honda}, \cite{Honda2}  this form can be replaced by an intrinsically harmonic 1-form lying in the same cohomology class and having the same Morse numbers. Thus, we may assume that $\Delta \alpha = 0 $ for some Riemannian metric on $Y$. Equip the 2n-manifold with the following 2-form:
\begin{equation}
\omega = \beta \wedge \alpha + \bar{\omega} + \left(\ast_{Y} \alpha\right)
\end{equation}
This 2-form is near-symplectic on $M$ and its singular locus is $Z_{\omega} = N \times \textnormal{Crit}(\alpha)$.
\end{example}


\section{Fibrations}
\subsection{Near-symplectic fibrations}
We recall the definition of broken Lefschetz fibrations  on dimension four.  On a smooth, closed 4-manifold $X^4$, a {\it broken Lefschetz fibration} or {\it BLF} is a smooth map to the 2-sphere, $f\colon X^4 \rightarrow S^2$, with two types of singularities:
\begin{enumerate}
\item isolated {\it Lefschetz-type} singularities, contained in the finite subset of points $B\subset X^4$, which are locally modeled by complex charts
$$ \Ctwo \map \C  \quad , \quad (z_1, z_2) \longmapsto z_{1}^{2} + z_{2}^{2}$$
\item {\it indefinite fold} singularities, also called {\it broken}, contained in the smooth embedded 1-dimensional submanifold $\Gamma \subset X^4 \setminus B$, which are locally modelled by the real charts
$$ \R^{4} \map \R^{2} \quad , \quad  (t,x_1,x_2,x_3) \longmapsto (t, x_{1}^{2} + x_{2}^{2} - x_{3}^{2})$$
\end{enumerate}
\begin{figure}[htbp]
\begin{center}
\resizebox{.41\textwidth}{!}{
 \includegraphics{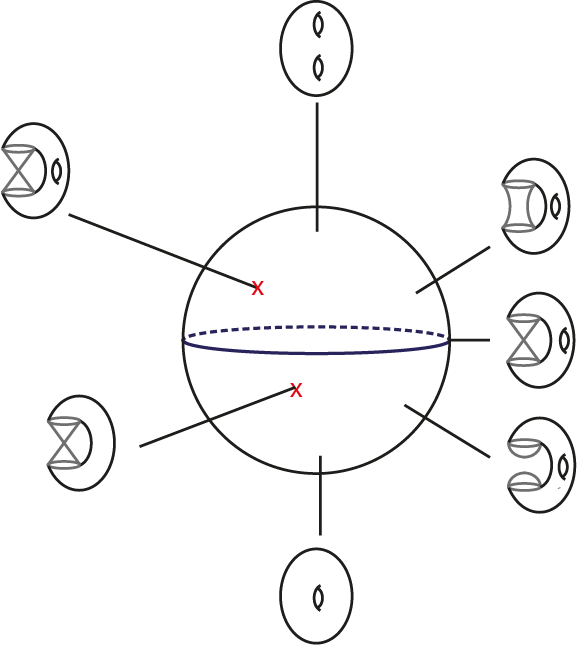} }
\caption{Example of a BLF with 1 circle of folds and 2 Lefschetz points}
\end{center}
\end{figure}
In \cite{ADK} these mappings were studied under the name of singular Lefschetz fibrations. It was shown that there is a relation between BLFs and near-symplectic manifolds. Up to blow ups, a near-symplectic 4-manifold $X$ can be decomposed into a BLF. The other direction is given by the following theorem.

\begin{theorem}[\cite{ADK}]
Given a BLF with singularity set $\Gamma \sqcup B$ on a closed oriented 4-manifold $X$, with the property that there is a class $\alpha \in H^{2}(X)$, such that it pairs positively with every component of every fiber, then $X$ carries a near-symplectic structure with zero-locus being equal to the set of broken singularities of $f$.
\end{theorem}

Our theorem \ref{theorem 1} shows a similar statement on $2n$ dimensions.  Now we define a map that will play an analogous role of a BLF two dimensions higher. This map is a submersion with folds and Lefschetz-type singularities. Notice that submersion with folds are stable if the map $f$ restricted to its fold set is an immersion with normal crossings \cite{Guillemin}. By stable we mean that any nearby map $\tilde{f}$ is identical to $f$ after changes of coordinates.

\begin{definition}\label{genBLF}
Let $M$ be a smooth, closed $2n$-manifold $M$ and $X$ a smooth, closed $(2n-2)$-manifold. By a {\it broken Lefschetz fibration} we mean a submersion $f\colon M \rightarrow X$ with two type of singularities:
\begin{itemize}
\item[(1)] ''extended'' {\it Lefschetz-type} singularities, locally modelled by
\begin{align}
\C^n &\rightarrow \C^{n-1}
\nonumber
\\
(z_1, \dots , z_n) &\rightarrow (z_1, \dots , z_{n-2}, z_{n-1}^2 + z_{n}^2)
\nonumber
\end{align}
\end{itemize}
These singularities are contained in codimension 4 submanifolds cross a Lefschetz singular point.  Singular fibres present an isolated nodal singularity, but nearby fibres are smooth and convex.
\begin{itemize}
\item[(2)] {\it indefinite fold} singularities, locally modeled by
\begin{align}
\R^{2n} &\rightarrow \R^{2n-2}
\nonumber
\\
(t_1, \dots, t_{2n-3}, x_1,x_2,x_3)& \mapsto (t_1, \dots, t_{2n-3},  -x_1^2 + x_2^2 + x_3^2)
\nonumber
\end{align}
\end{itemize}
The fold locus is an embedded codimension 3 submanifold, and we denote it by $\Sigma$.  Singular fibres present a nodal singularity, but this time crossing $\Sigma$ changes the genus of the regular fibre by one.
\begin{figure}[htbp]
\begin{center}
\resizebox{.45\textwidth}{!}{
 \includegraphics{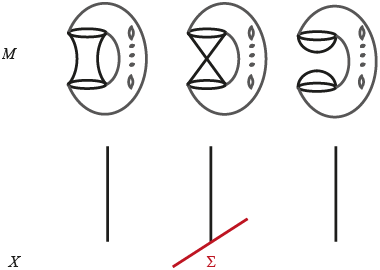} }
\caption{Fibres: regular (left, right) and singular (middle)}
\end{center}
\end{figure}

\noindent If we consider the total space to be near-symplectic with regular fibers being symplectic, and the fold locus being $Z_{\omega}$, then we will refer to the previous map $f\colon M \rightarrow X$ as a {\it near-symplectic fibration}.
\end{definition}
%
\subsection{Proof of theorem \ref{theorem 1}} \label{proof theorem 1}
\noindent {\bf{Step 1: Constructing the local 2-form}}
\\
First we want to define the local near-symplectic form near the singular sets $\Sigma \sqcup C$, where $\Sigma$ denotes the singularity set of folds and $C$ the set of extended Lefschetz-type singularities. We begin by defining a singular symplectic form vanishing at $\Sigma$, and then we pull back the symplectic form of the base. Let $(t_1, \dots , t_{2n-3}, x_1,x_2,x_3) $ be coordinates around a fold point $p\in \Sigma$ of index 1, locally modeled by $\tilde{f}\colon (t, x) \mapsto (t_1, \dots , t_{2n-3}, -x_1^2 + \frac{1}{2} (x_2^2 + x_3^2))$. Since the fibres are 2-dimensional, we can take a similar local model as the near-symplectic forms on 4-manifolds.  Define the following 2-form on a piece of the tubular neighbourhood of $\Sigma$ containing $p$:
\begin{equation}\label{form on fibers}
\tau_{p} = d \left( \chi(t) x_{1} (x_{2} dx_{3} - x_{3} dx_{2}) \right)
\end{equation}
This 2-form is closed, vanishes at the singularity set,  is non-degenerate outside $\Sigma$, and evaluates positive on the fibres. 
Here, $\chi(t)$ is a smooth cut-off function depending on coordinates on $Z$.  This cut-off function will help us in the gluing process when summing up the 2-forms $\tau_{p_{i}}$ to build a local 2-form on the whole tubular neighbourhood of $\Sigma$. We sum up the forms $\tau_{p_{k}}$ over a finite cover of $\Sigma$, and pullback the symplectic form from the base.  We obtain
\begin{equation}
\omega_{A} = \sum_{p_{k}} \tau_{p_{k}} + f^{*}\omega_X
\end{equation}
This closed 2-form is defined on the tubular neighbourhood of $\Sigma$.  It is non-degenerate outside $\Sigma$ and posiitve on the fibres.  At the degeneracy points, $K_p = N_p\Sigma \oplus \varepsilon_{p}$ is of dimension four, where $\varepsilon = \ker(f^{*}\omega_X) \subset T\Sigma$. 

Around the elements of $C$, where $f$ is given by $(z_1, \dots, z_{n}) \mapsto (z_1, \dots, z_{n-2}, z_{n-1}^2 + z_{n}^2)$, we can choose disjoint neighbourhoods $B_k$ such that $\omega_0|_{B_k} = \omega_{\mathbb{C}^{n-1}}$. Since near $C$ we are in a situation similar to a Lefschetz fibration we can proceed as in \cite{ADK},\cite{Gompf}. For any $v_1, v_2 \in T_p F$, we get $\omega_{0}|_{B_k}(v_1,v_2)>0$ away from the singularity. The symplectic form $\omega_A|_{B_k}$ can be extended to the fibre $F_q$ as a symplectic form for all $q \in f(B_k) \subset X$. 
\\
\\
\noindent {\bf{Step 2: Extension over the neighbourhoods of the fibres}}
\\
In this step we want to construct local 2-forms on the neighbourhood of the fibres.  We extend the 2-form to a local model over the neighbourhood of the fibres, such that it agrees with $\omega_{A}$ near $\Sigma \sqcup C$. Let $U$ be the tubular neighbourhood of $\Sigma \sqcup C$.  Choose a closed 2-form $\zeta \in \Omega^{2}(M)$ with a class being represented by $\alpha$. Since $\alpha|_{\Sigma}= \omega_X|_{\Sigma} \in H^2(\Sigma)$, over $U$ there exists a 1-form $\bar{\mu} \in \Omega^{1}(U)$, such that $\omega_{A} - \zeta = d\bar{\mu}$.  We extend now $\bar{\mu}$ to an arbitrary 1-form  on the manifold, $\mu \in \Omega^{1}(M)$,  supported in a neighbourhood $W$ of $U$.  By substituting $\eta = \zeta + d\mu$ on $U$, we can regard $\eta$ to be $\omega_{A}$ when restricted to $U$.

By assumption, we have a positive pairing, $\langle \alpha, F \rangle > 0$ over each component of the fibre, $[\eta] = \alpha$, and the fibres have a symplectic form $\sigma_F$.  We equip the fibres with a closed singular 2-form $\sigma_q$ such that 
\begin{itemize}
 \item [(a)] $\sigma_{q}|_{F_{q} \cap U_{1}} = \eta$. By restricting $\sigma_q$ to $U$, this 2-form is near-symplectic, since  $\eta|_{U} = \omega_{A}$. The form $\sigma_{q}$ is defined on the fibre, so $\sigma_q|_{F_{q} \cap U_{1}}$ is near-symplectic.
\item [(b)] $\sigma_{q}|_{F_{q}}$ is positive on $F_q$, where the fibre is smooth.  This can be seen by considering two subsets of the fibre. Take a small open neighbourhood around the singularity and a second larger one covering the rest of the fibre.  On the first neighbourhood around the singularity, the 2-form $\omega_A$ evaluates positively except at the singular point. On the second subset where the fibre is smooth, the area form of $F_q$ evaluates positively. 
\item [(c)] $\int_{F} \sigma_{q} = \langle \alpha , F \rangle > 0$
, since $[\sigma_{q} - \eta|_{F_{q}}] = 0$ in $H^{2}(F_{q}, F_{q}\cap U_{1}) \stackrel{PD}\simeq H_{0}(F_{q}, F_{q} \cap U_{1})\simeq 0$,  assuming $F_{q}$ is connected.  Then $(\sigma_{q} - \eta|_{F_{q}})$ is exact in $F_{q}\cap U_{1}$, that is $[\sigma_{q}]= [\eta] = \alpha$.
\end{itemize}

Now, we describe some properties of the neighbourhood of the fibres in order to extend the 2-form. For any $q \in X$ we can find a tubular neighbourhood $V_q$ of the fibre $F_q$ and neighbourhoods $U_{2} \subset U_{1} \subset U$ of the fold singularity set $\Sigma$.   A $q\in X$ can be engulfed by an $m$-disk $D^{m}$. Around a fibre $F_{q}$, take $f^{-1}(D^{m}) = V_{q}$.  
After removing a small neigbourhood of the critical set, we have that $V_q\setminus (V_q \cap U_{2})$ is diffeomorphic to $D^{m-1} \times (F_{q} \setminus (F_{q} \cap U_{2}))$. This follows from the Ehresmann theorem, since we have a nice smooth map locally without critical points. 

To extend the 2-form on the neighbourhood the fibre, we build a smooth map 
$$\pi\colon V_q \rightarrow V_q$$
by interpolating between two maps:
\begin{itemize}
\item[(i)] Close to the singular point of the fibre inside the neighbourhood $V_q \cap U_1$, we use the identity map, so that $\pi$ is $\textnormal{id}_{V_q \cap U_1}$.  Since  $V_q$ is a neighbourhood of a fibre $F_q$ and  $V_q \cap U_{1}$ retracts to $F_{q} \cap \Sigma$, we  want that $\pi$ maps down to the piece of the fibre close to the singularity together with the intersection of the neighbourhoods $V_q$ and $U_1$.  That is, 
$$\mathtt{Im}(\pi) \subset F_q \cup (V_q \cap U_{1}) $$

\item[(ii)] Farther away from the singular region, that is on the smooth part $F_q \setminus (F_q \cap U_2)$, we use the projection map $\mathtt{pr}\colon V_{q} \setminus (V_{q} \cap U_{2}) \rightarrow F_q \setminus  F_q \cap U_{2}$ that comes from the product structure. 
\end{itemize}
\noindent We use the map $\pi$ to construct a near-symplectic form $\beta$ on $V_q$.  With $\pi$, we pull back the 2-form $\eta$ on $V_q \cap U_1$ and the 2-form $\sigma_q$ on $F_q$ 
$$\beta = \pi^{*}\sigma_{q} \, + \,  \pi^{*}\eta $$
This 2-form has the following features:
\begin{enumerate}
 \item $d\beta=0$ and $[\beta] = \alpha|_{V_{q}}$
 \item $\beta|_{V_{q}\cap U_{2}} = \eta$
 \item there exists a 1-form $\mu_{q}$ on $V_{q}$, such that $\beta - \eta = d\mu_{q}$, since $[\beta - \eta] = 0 $ in $H^{2}(V_{q}, V_{q}\cap U_{2}) \simeq H^{2}(F_{q}, F_{q}\cap U_{2})$.  Thus, on $V_q$
$$\beta_{q} = \eta + d\mu_{q}$$
 \item $\beta_{q}|_{F_{q}} >0$ restricts positively to the fibre for every regular point $q\in V_{q}$.
\end{enumerate}

%
%
%
%
\noindent {\bf{Step 3: Patching into a global form}}
\\
We expand the near-symplectic form over the whole manifold $M$.  Since our base is compact, we can find a finite subset $Q\subset X$ and choose a finite cover $\mathcal{D}$ with open subsets $\left( D_{q} \right)_{q\in Q}$, such that $f^{-1}(D_{q}) \subset V_{q}$ for each $q\in X$.  Consider a smooth partition of unity $\rho\colon X \rightarrow [0,1]$, $\sum_{q\in Q} \rho_{q} =1$, subordinate to the cover $\mathcal{D}$ with $\mathtt{supp}(\rho_{q}) \subset D_{q}$. We build a global 2-form $\Omega$ on $M$ by patching the local 1-forms $\mu_{q}$ previously defined on $V_{q}$. Thus, we define the following closed 2-form
\begin{equation}
\Omega = \eta + d\left( \sum_{q\in Q} (\rho_{q}\circ f) \mu_{q}  \right)
\end{equation}
Since $f$ is constant on the fibres, the 1-form $d((\rho_{q} \circ f) \mu_{q})=0$ when evaluated on the vectors tangent to the fibre.   From the second step, $\eta$ agrees with $\omega_{A}$ when restricted to $U$.  Let $\bar{U}_{2}$ be the intersection of all neighbourhoods $U_{2}$ for all $q\in Q$, that is, $\bar{U}_{2} = U_{2} \bigcap_{q\in Q} f^{-1}(D_{q})$. The 2-form $\Omega$ agrees with $\eta$ when restricted to $\bar{U}_{2}$, so it agrees with the local model of $\omega_{A}$ at $U_{2}$. Thus, $\Omega$ is globally well-defined over $M$.

\noindent The 2-form $\Omega$ restricts to a fibre $F_{q}$ in the following way
\\
\noindent $\Omega|_{F_{q}} =  \eta|_{F_{q}} + \sum_{q\in Q} \rho\circ f(p)  d\mu_{q}|_{F_{q}}
= \sum_{q\in Q} \rho\circ f(p) (\eta + d\mu_{q})|_{F_{q}}$
\\
\indent $\quad = \sum_{q\in Q} (\rho\circ f(p)) \, \beta_q|_{F_{q}} $

This is a convex combination of near-symplectic 2-forms. On each fibre, $\Omega$ is closed, positive outside the singular locus, and degenerates at $\Sigma$, inducing a symplectic structure on each fibre outside the singularities.
%
%
%
%
\\
\\
\noindent {\bf{Step 4: Positivity on vertical and horizontal tangent subspaces}}
\\
\noindent  \noindent  To conclude the global construction, we can apply a similar argument as in the symplectic case. 
The 2-form $\Omega$ is positive on the vertical tangent subspaces to the fibre $\ker df(p) = T_{p} F \subset T_{p} M $,  outside the singularity set. To guarantee positivity on the horizontal spaces, we multiply the pullback from the symplectic form of the base by a sufficiently large real number $K>0$ to obtain the 2-form 
\begin{equation}
\omega_{K} = \Omega + K\cdot f^{*}\omega_{X}
\end{equation}
If we restrict $\omega_{K}$ to the vertical tangent subspaces to the fibre, it agrees with $\Omega$.  The 2-form $\omega_K$ defines a near-symplectic structure on $M$.  \qed
%

\subsection{Examples}\label{example fibration}
\subsubsection{Pullback bundle}
We can obtain examples of near-symplectic manifolds and near-symplectic fibrations via a pullback bundle construction. Let $n>2$, $M$ and $X$ oriented, closed, $(2n-2)$-manifolds, $B$ an oriented, closed, connected manifold of dimension $(2n-4)$, $f$ and $g$ smooth mappings, and $W = \lbrace (x,m) \in X\times M \mid f(x) = g(m) \rbrace$.
\smallskip
\begin{align}
\xymatrix{
W \ar[d]_{\tilde{g}}\ar[r]^{\tilde{f}}        &  M \ar[d]^{g} \\
X \ar[r]_{f}                                &    B
}
\nonumber
\end{align}
Before going to the near-symplectic case, we briefly comment on the symplectic one.  A theorem from Thurston tells us, that if $g$ is a compact symplectic fibration over a closed connected symplectic manifold $B$, and there is a class $\alpha \in H^2(M)$ such that $\iota^{*} \alpha = \left[ \sigma_b \right] \forall b\in B$, where $\sigma_b \in \Omega^2 (F_b)$ is the canonical form of the fibre, then $M$ is symplectic.  We can pullback this information to $W$ via $\tilde{f}$, and obtain a class $\tilde{\alpha} = \tilde{f}^{*} \alpha \in H^2 (W)$ with the same property. Thus, we only need $X$ to be symplectic in order that $W$ is a symplectic manifold via the induced map $\tilde{g}$. Now we discuss the near-symplectic scenario.  

Throughout these examples we assume that the critical set of $g$ form regular points for $f$, so that $f$ behaves like a bundle near the critical sets by the Ehresmann theorem (whenever there is a critical set for $g$).  The first example follows from theorem 1. If $g$ is a BLF (thus $\tilde{g}$ a generalized BLF), and $X$ is symplectic, then $W$ is near-symplectic via $\tilde{g}$ assumming that the cohomological condition of theorem 1 is satisfied. A second case appears when the base $X$ is near-symplectic. Keeping a vertical view of the diagram, now we do not consider $g$ and $\tilde{g}$ to be BLFs.  The following proposition explains this situation.
\begin{proposition}
Let $g\colon M \rightarrow B$ be a compact symplectic fibration with symplectic total space $M$ and  $(X, \omega_{X})$ a closed, near-symplectic manifold over a closed connected symplectic base $B$ of codimension 2.  Let $W$ be the pullback bundle as defined in the previous parragraph.  Then, $W$ carries a near-symplectic structure induced by $\tilde{g}\colon W \rightarrow X$. 
\end{proposition}

\begin{proof}
Let $\Gamma$ be the singular locus of $\omega_{X}$, that is a codimension 3 submanifold in $X$.  Its preimage under $\tilde{g}$ is a surface bundle 
, and we will denote by $Z$ its total space. This bundle will become the singular locus of the near-symplectic form of $W$.  Let $\mathcal{U}$ be the tubular neighbourhood of $\Gamma$ and let $E = \tilde{g}^{-1}(\mathcal{U})$. 
$E$ is a surface bundle. We will also consider a small tubular neighbourhood $\bar{E}$ inside $E$.

Now we construct a closed 2-form $\tilde{\eta} \in \Omega^2(W)$ such that it is positive on the fibres of $\tilde{g}$ in $W$ and whose wedge power $\tilde{\eta}^{n-1}$ is zero on $E$. Since $g$ is a symplectic fibration, we have a cohomology class $\alpha \in H^2(M)$ that pairs pairs positively with the fibre class. We choose $\tilde{\eta}$ such that $[\tilde{\eta}] = \tilde{f}^{*} \alpha \in H^2 (W)$ with $\iota^{*}\tilde{\alpha} = \tilde{f}^{*} \left[ \sigma \right]$. Secondly, as $\bar{E}$ and $E$ are cohomologically $(2n-3)$-dimensional, we can select $\tilde{\eta}$ with the property that $\tilde{\eta}^{n-1}|_{\bar{E}} = 0$.

Let $U_k$ be contractible open subsets of a cover of $B$ with trivializations $\phi_k$, such that $\phi_k \circ \phi^{-1}_j$ are symplectomorphisms over $U_k \cap U_j$. We bring these neighbourhoods to $W$ as $(\tilde{g} \circ f)^{-1}(U_k) = \tilde{U}_k$. Define $\psi_k: = (\textnormal{proj}\circ \tilde{\phi}_k \circ \tilde{f}) \colon \tilde{U}_k \rightarrow F$.  Over $\tilde{U}_k$ there is a 1-form $\mu_k$ such that $d\mu_k = \psi^{*} \tilde{\sigma}_k - \tilde{\eta}_k$, since $\left[ \tilde{\eta} \right] = \tilde{f}^{*}|_{F}\left( \alpha \right) = \left[ \psi^{*} \tilde{\sigma} \right]$.

The rest of the proof follows similarly as in step 3 and 4 of theorem's 1 proof. Choose a partition of unity $\rho \colon W\rightarrow [0,1]$ in such a way that its open subsets do not touch $\bar{E}$, and with it define a closed 2-form $\beta = \tilde{\eta} + \sum_k \rho_{k} d \mu_k$ on $W$.  This form has the properties that: $\beta|_{\bar{E}} = \tilde{\eta}|_{\bar{E}}$ and $\beta|_{F} = \sigma_b$, where $\sigma_b$ is the form of the fibre $F_b$.  Finally, we build up our global form by adding $\tilde{g}^{*}\omega_{X}$.  If $K$ is a sufficiently large positive real number, then we have a closed 2-form, which is non-degenerate away from $Z$
$$\omega_{K} = \beta + K \cdot \tilde{g}^{*} \omega_X$$
\end{proof}
\subsubsection{Near-symplectic manifolds coming from BLFs}
Broken Lefschetz fibrations also provide ways to obtain near-symplectic fibrations on $2n$-manifolds over near-symplectic $(2n-2)$-manifolds. Let $g \colon M\rightarrow B$ be a BLF as defined previously with singular fold set $\Sigma_{\tilde{g}}$, where $M$ is near-symplectic of $\dim(M) \geq 4$ and $B$ is a closed, connected, symplectic manifold of $\dim(B) \geq 2$. Furthermore, consider $(X, \omega_X)$ to be a symplectic manifold of $\dim(X) \geq 4$. Assume that there is a class $\alpha \in H^{2}(M)$ such that $\langle \alpha, F \rangle >0$ and $\tilde{\alpha}|_{\Sigma_{\tilde{g}}} = \omega_X|_{\Sigma_{\tilde{g}}}$. Then, $W$ is near-symplectic via a generalized BLF $\tilde{g}$. 

If both $f\colon X\rightarrow B$ and $g \colon M\rightarrow B$ are two BLFs, then we require the intersection of their critical images to be transversal in $B$, but not necessarily disjoint. In that case, it follows from a standard differential topological argument that $W$ is a $2n$-dimensional manifold. The maps $\tilde{f}$ and $\tilde{g}$ become near-symplectic fibrations, carrying the same type and number of fold and Lefschetz-type singularities as $f$ and $g$ respectively. Around a critical point in $f^{*}M$, the maps $\tilde{f}$ and $\tilde{g}$ are locally modelled by coordinate charts $\varphi$ and $\pi$ respectively defined as 
\begin{align}
\varphi \colon \R^{2n} &\rightarrow \R^{2n-2} \quad &, \quad  (r_{1}, \dots , r_{2n}) &\mapsto (r_1^2 + r_2^2 - r_3^2, r_4, \dots , r_{2n}) 
\nonumber
\\
\pi \colon \R^{2n} &\rightarrow \R^{2n-2} \quad &, \quad (r_{1}, \dots , r_{2n})& \mapsto (r_1, \dots , r_{2n-3}, -r_{2n-2}^2 + r_{2n-1}^2 + r_{2n}^2)  
\nonumber
\end{align}
Assume the cohomological condition on the class $\tilde{\alpha} \in H^{2}(W)$ as above, and denote by $\Gamma$ the singular locus of $\omega_{X}$, and $\Sigma$ the singularity set of $\tilde{g}$. The mapping $\tilde{g}$ becomes a a near-symplectic fibration over a near-symplectic base $(X, \omega_{X})$, if $\tilde{g}^{-1}(\Gamma) \not\subset \Sigma$ in $W$.  This construction gives 2 generalized BLFs, one for each pullback mapping.

\begin{rem}
Even though lemma 3.5 together with step 1 tell us that the fold map pulls back symplectomorphically, other types of singularities might need a different treatment.  For instance, if we would like to consider deformations of near-symplectic fibrations, in a similar fashion as Lekili \cite{Lekili}, then it would be necessary to consider all stable singularities of maps from $\R^{2n}$ to $\R^{2n-2}$. For maps going from a 6 dimensional source to a 4 dimensional target, there are 4 stable singularities: folds, cusps, swallowtails, and butterflies \cite{ArnoldSing, Guillemin}. For higher dimensions the list becomes longer and more complicated. 
\end{rem}

\section{Geometry of the Singular Locus}
In this section we study the geometry around the singular locus induced by the near-symplectic form. First, we show that the singular locus carries a natural Hamiltonian structure. Then we show that if $Z$ admits a stable Hamiltonian structure, so does its normal sphere bundle $Z\times S^2$, in the case where the normal bundle is trivial. In the second part, we describe the splitting property of the  the normal bundle following from a near-symplectic structure, similar to the 4-dimensional case.  Then, we give a neighbourhood-type theorem. As a corollary, we find a local Darboux-type statement for near-symplectic forms. . 
\subsection{Stable Hamiltonian Structures}
We present the next definitions as exposed by Cielebak and Volkov \cite{Cieliebak}.                                               
\begin{definition}
A {\it Hamiltonian structure} (HS) on an oriented $(2n-1)$-dimensional manifold $M$ is a closed 2-form $\Omega$ such that $\Omega^{n-1}\not= 0 $ everywhere.  Associated to $\Omega$ is its 1-dimensional {\it kernel distribution}  $\ker(\Omega):= \lbrace v\in TM \mid \iota_v \Omega = 0 \rbrace$.  We orient $\ker(\Omega)$ using the orientation on $M$ together with the orientation on the local transversal to $\ker(\Omega)$ given by $\Omega^{n-1}$.  

A {\it stabilizing 1-form} for $\Omega$ is a 1-form $\lambda$ such that 
\begin{enumerate}
  \item $\lambda \wedge \Omega^{n-1}>0$, and

  \item $\ker(\Omega) \subset \ker(d\lambda)$
\end{enumerate}
A Hamiltonian structure $\Omega$ is called {\it stabilizable} if it admits a stabilizing 1-form $\lambda$.  A {\it stable Hamiltonian structure (SHS)} is the pair $(\Omega, \lambda)$.
\end{definition}

A SHS $(\Omega, \lambda)$ induces a canonical {\it Reeb vector field} $R$ generating $\ker(\Omega)$ and normalized by $\lambda(R)=1$.  Not that if $(\Omega, \lambda)$ is a SHS, then $(\Omega, -\lambda)$ is a SHS inducing the opposite orientation.

\begin{example}
\quad
\begin{enumerate}
 \item {\it Contact manifolds}: $(M, \lambda)$ is a contact manifold, $R$ is the Reeb vector field, and $\Omega =\pm d\lambda$.
\item {\it Mapping tori}: $M:= W_{\phi} = \R \times W \slash (t,x) \sim (t+1, \phi(x))$ is the mapping torus of a symplectomorphism $\phi$ of a symplectic manifold $(W, \bar{\omega}), R = \frac{\partial}{\partial t}, \lambda = dt$, and $\Omega$ is the form on $M$ induced by $\bar{\omega}$.  Note that $d\lambda= 0$, so $\ker(\lambda)$ defines a foliation.  Notice that $W_{\phi} = [0,1] \times W \slash (0,x) \simeq (1, \phi(x))$
\item {\it Circle bundles}: $\pi\colon M \rightarrow W$ is a principal circle bundle over a symplectic manifold $(W, \bar{\omega})$, $R$ is the vector field generating the circle action, $\lambda$ is the connection form, and $\Omega = \pi^{*}\bar{\omega}$.  
\end{enumerate}
\end{example}

The next example follows directly from the definition of a near-symplectic form.
\begin{proposition}
A near-symplectic structure induces a Hamiltonian structure on its singular locus $Z_{\omega}$.
\end{proposition}

\begin{proposition}
Let  $(Z\times \R^3, \omega)$ be a near-symplectic manifold with singular locus $Z \times \{0\}$, where $Z$ is an oriented $(2n-1)$-manifold.  If $\varepsilon$ is a stabilizing 1-form for $\omega_Z$ on $Z$, then the normal sphere bundle $Z\times S^2$ has a stable Hamiltonian structure.
\end{proposition}

\begin{proof}
By assumption, we have that $\varepsilon \wedge \omega_Z^{n-2} >0$ on $Z$ and $\ker(\omega_Z) \subset \ker(d\varepsilon)$.  Let $\sigma_{S^2}$ be the symplectic form of $S^2$. 
The boundary of a piece of the tubular neighbourhood $\partial (Z\times B^3) = Z\times S^2$ can be equipped with a Hamiltonian structure by 
\begin{equation}
\bar{\omega} = \omega_Z + \sigma_{S^2}
\end{equation}
This is a closed 2-form of maximal rank on $Z\times S^2$, since $\bar{\omega}^{n-1} = \omega_Z^{n-2} \wedge \sigma_{S^2} >0$.  The stabilizing 1-form on $Z\times S^2$ is defined by $\lambda = \varepsilon$. 
We have
$$\lambda \wedge \bar{\omega}^{n-1} = \varepsilon \wedge (\bar{\omega}^{n-2} \wedge \sigma_{S^2})>0 $$
This shows the first condition of a SHS.  Now, for the second property observe that 
$$ \ker(\bar{\omega}) =\lbrace v\in TM \mid \iota_v \bar{\omega} =  \iota_v (\omega_Z + \sigma_{S^2}) = 0 \rbrace \simeq \mathcal{E} = \ker(\omega_Z) $$
%
%
%
In this case $\ker(\bar{\omega}) \subset \ker(d\lambda)$. The pair $(\bar{\omega}, \lambda)$ is a stable Hamiltonian structure for $Z\times S^2 \subset (M, \omega_{\textnormal{ns}})$.
\end{proof}

\subsection*{Stable Hamiltonian in BLF case}
\begin{proposition}
Let $(Z, \xi_Z = \ker(\alpha_Z))$ be a contact manifold of dimension $2n-1$, and $(Z\times \R, \omega_B = d(e^t \alpha_Z))$ its symplectization.  Let $f\colon Z\times \R^3 \rightarrow Z\times \R$ be a broken Lefschetz fibration. The total space $Z\times \R^3$ is near-symplectic inducing a stable Hamiltonian structure on $Z\times S^2$. 
\end{proposition}
\begin{proof}
We now equip $M = Z\times \R^3$ with a near-symplectic form along the lines of \cite{ADK} and theorem 1. Over the regular neighbourhood of $Z$, using the coordinates $(x_i)$ of the fibre,   define the 2-form  
\begin{equation}\label{tau}
  \tau = d (x_1 (x_2 dx_3 - x_3 dx_2))
\end{equation}	
We obtain a closed 2-form, positive on the fibres and non-degenerate outside $Z$. Define the 2-form $\omega \in \Omega^2(Z\times \R^3)$ as 
$$\omega = \tau + f^{*} \omega_B$$
At the points where $\omega^n = 0$ we have a 4-dimensional kernel $K = \lbrace v\in T_p M \mid  \omega_p (v,\cdot) = 0\rbrace \simeq \varepsilon \oplus TY^3$, where $\varepsilon = \ker(f^{*}\omega_B)$. 
The 2-form $\omega$ defines a near-symplectic structure on $Z\times \R^3$.

Let $\mathcal{U}$ be the tubular neighbourhood of $Z$ in $M$ and $\sigma_{S^2}$ the area form of $S^2$. Define on the boundary of $\mathcal{U}$ the 2-form
\begin{equation}
\bar{\omega} = d\alpha_Z + \sigma_{S^2}
\end{equation}
The contact form $\alpha_Z$ will work as the stabilizing 1-form $\lambda = \alpha_Z$. 
A simple computation shows that 
$$\lambda \wedge \bar{\omega}^{n-1} = \alpha_Z \wedge d\alpha_{Z}^{n-2} \wedge \sigma_{S^2} > 0$$
Moreover, since $\ker(\bar{\omega}) \simeq \varepsilon \simeq \ker(d\alpha_Z)$, the second property is also satisfied.
Hence, the pair $(\bar{\omega}, \alpha_Z)$ defines a stable Hamiltonian structure on the boundary of the singular locus $Z\times S^2$. 
\end{proof}

\subsection{Normal bundle of $Z$}
In this section, we will first show that the definition of near-symplectic form reflects properties on the normal bundle of the singular locus similar to dimension 4. In particular, we obtain a splitting of the normal bundle $N_Z$ into two subbundles. 
%
%
%
%
%
Let $K:= \epsilon \oplus N_Z$ be defined by $\epsilon = \ker(\omega_Z)$ and the normal bundle of $Z$, the singular locus of $\omega$. Fix a metric $g$ on $K$ such that $\omega|_K$ is self-dual. 
%
%
Identify the intrinsic normal bundle $N_{Z}$ with the complement $(TZ)^{\perp}$ using the metric $g$. From the transversality of $\omega$, the image of the intrisic gradient $D_K := \nabla \omega|_{K}$ is 3-dimensional. In fact, we have that $\textnormal{Im}(D_k) = \Lambda^{2}_{+} K^{*}$. Thus we have a natural identification with the bundle of self-dual 2-forms. 
%
%
This implies that $D_K$ defines an isomorphism
$$N_{Z} \rightarrow \Lambda^{2}_{+} K^{*}$$
Let $X = \frac{\partial}{\partial z_0}$ be unit vector field defined on the line $\ker(\omega|_{Z}) \subset TZ$. The interior derivative defines a bundle isomorphism  
\begin{align}
\Lambda^{2}_{+} K^{*} &\rightarrow N_{Z}^{*} 
\nonumber
\\
\beta &\mapsto \iota_X \beta 
\nonumber
\end{align}
Its inverse $N_{Z}^{*} \rightarrow \Lambda_{+}$ is given by $\nu \mapsto \zeta \wedge \nu + \ast(\zeta \wedge \nu)$, where $\zeta$ is a 1-form that is non-vanishing on $\epsilon$. Using the metric we can define an isomorphism $N_{Z}^{*}\rightarrow N_{Z}$. The endomorphism 
$$F\colon N_{Z} \rightarrow N_{Z}$$
defined by the composition 
$$N_{Z} \xrightarrow{D_K}  \Lambda_{+} \xrightarrow {\iota_X} N_{Z}^{*} \xrightarrow{\, g \,} N_{Z}$$
is a self-adjoint, trace-free automorphism as in dimension 4 \cite{Perutz, Honda3}. The matrix $A$ representing this map is symmetric and trace-free. Consequently, at each point $p\in Z$, $A$ has three real eigenvalues, two of them positive and one negative, following the sign convention used in low dimensions \cite{Perutz, Honda3, Taubes}. %
We obtain a splitting of the normal bundle in 2 eigensubbundles defined by the negative and positive eigenspaces 
$$N_{Z} \simeq L_{-} \oplus L_{+} $$
Here $L_{-}$ is a rank 1 bundle, locally trivial, and $L_{+}$ is a rank 2 bundle orthogonal complementary to $L_{-}$.  After a choice of basis the linear map $F$ can be represented by a trace-free symmetric matrix $A = A_{+} \oplus A_{-}$, where $A_{+}$ is a $2\times 2$ positive-definite matrix, and $A_{-} < 0$.
\subsection{Proof of Theorem 2}\label{neighbourhood thm}
\noindent {\bf Step 1: Family of near-symplectic forms}
\\
Define $\omega_t = (1-t)\omega_0 + t \cdot \omega_1$.  We want to show that this is a family of near-symplectic forms. The closedness property follows from the fact that this family is a linear combination of closed 2-forms.  The symplectic subspaces defined by $\omega_{Z_{0}}$ and $\omega_{Z_{1}}$ are the same on $TZ_0 \simeq \textnormal{Symp}_{0}\oplus \varepsilon_0$ and $TZ_1 \simeq \textnormal{Symp}_{1}\oplus \varepsilon_1$. This defines the same complementary line bundle $\varepsilon = \ker(\omega_{Z_{0}}) = \ker(\omega_{Z_{1}})$. 
\\

\noindent The kernels $K_{0} \simeq \varepsilon \oplus N_{Z_{0}} $ and $K_{1} \simeq \varepsilon \oplus N_{Z_{0}}$ are 4-dimensional. 
Interpolating between $\omega_0$ and $\omega_1$ leaves $\dim(K_t) = 4  \quad \forall t$.  
Thus, up to scaling the intrinsic gradients $D_{K_{0}}:= \nabla \omega|_{K_0}$ and $ D_{K_{1}}:= \nabla \omega|_{K_1}$ agree and so their images. 
Hence, it follows that at a point $p=0$ in $Z$ we have that $\omega_{t}^{n}=0$.  Notice that this property  can also be computed directly by looking at the expansion 
\begin{align}
\omega_t^n(0) &= c_n(t) \omega_{0}^{n} (0) +   c_{n-1} (t)  {n\choose 1} \cdot \omega_{0}^{n-1} \wedge \omega_1  (0) + c_{n-2}(t) {n\choose 2}  \omega_{0}^{n-2}  \wedge  \omega_{1}^{2}(0)   + \dots 
\nonumber
\\
&+ c_{0}(t) \omega_{1}^{n} (0)
\nonumber
\end{align}
\noindent where $c_k(t)= (1-t)^k\cdot t^{n-k}$, with $k\in \mathbb{Z}, k\in [0,n]$.   In the previous expression all terms vanish, since each of them necessarily takes four vectors from $K_t$.  

Now we show that $\omega_t^n$ is non-negative.  Let  $v$ a vector in $N_Z$ and $s\in \R$.  Consider the Taylor expansion around $p\in Z$. 
\begin{align}
&\omega_t^n(s\cdot v) 
\nonumber
\\
&= \underbrace{\omega^{n}_{0}(0)}_{=0} \quad         + s\cdot \underbrace{\omega_{0}^{n-1}}_{=0} \wedge \nabla_v \omega_0                                                                                                + s^2 \cdot \omega^{n-2}_{0} \wedge (\nabla_v \omega_0)^2 + \dots +  \underbrace{\omega_{0}^{k} \wedge \omega_{1}^{n-k}(0)}_{=0}
\nonumber
\\
                                    & + s\cdot \underbrace{\omega_{0}^{k-1} \wedge \omega_{1}^{n-k}}_{=0}\wedge \nabla\omega_0  +  s\cdot \underbrace{\omega_{0}^{k} \wedge \omega_{1}^{n-k-1}}_{=0}\wedge \nabla\omega_1 
\nonumber
\\
                                    &  + s^2\cdot (\omega_{0}^{k-2}\wedge\omega_{1}^{n-k} \wedge (\nabla_v\omega_{0})^2 +   \omega_{0}^{k-1}\wedge\omega_{1}^{n-k-1} \wedge (\nabla_v\omega_{0}\wedge \nabla_v\omega_{1}) 
\nonumber
\\
                                    & +   \omega_{0}^{k}\wedge\omega_{1}^{n-k-2} \wedge (\nabla_v\omega_{1})^2 )+ \dots 
\nonumber
\\
&= \underbrace{\omega^{n}_{1}(0)}_{=0} \quad         + s\cdot \underbrace{\omega_{1}^{n-1}}_{=0} \wedge \nabla_v \omega_1                 + s^2 \cdot \omega^{n-2}_{1} \wedge (\nabla_v \omega_1)^2 + \dots
\nonumber
\end{align}
The terms of the form $\omega_{0}^{k} \wedge \omega_{1}^{n-k}$ for $k\in\lbrace 0, \dots , n\rbrace$  vanish identically as explained in the previous paragraph. The linear terms of the form $\omega_{0}^{k-1} \wedge \omega_{1}^{n-k+1} \wedge \nabla_v \omega_i$ for $i=0,1$ are also zero, since from the $2n-2$ vectors $v_i$ which are allocated in $\omega_{0}^{k-1} \wedge \omega_{1}^{n-k+1} (v_1, \dots, v_{2n-2})$,  at least 2 of those vectors should come from $K_t$. This leaves us with the following expression with leading terms of the order $s^2$
\begin{align}
&\omega_t^n(s\cdot v)
\nonumber
\\
&=s^2 \cdot \big{(}\omega_{0}^{n-2} (0) \wedge (\nabla_v \omega_0)^2   +  \dots 
\nonumber
\\
                                      &+ \omega_{0}^{n-2} \wedge \left(  \nabla_v \omega_0 \wedge \nabla_v \omega_1  \right) + \dots + \omega_{0}^{n-3}\wedge \omega_1 (\nabla_v \omega_{0})^2 + \dots 
\nonumber
\\
                                      &+ \omega_{0}^{n-2} \wedge \left(   \nabla_v \omega_1  \right)^{2} + \dots + \omega_{0}^{n-3} \wedge \omega_1 \left( \nabla_v \omega_0 \wedge \nabla_v \omega_1\right) + \dots + \omega_{0}^{n-4}\wedge \omega_{1}^{2}\wedge (\nabla_v \omega_0)^2 + \dots 
\nonumber
\\
                                    &+ \omega_{0}^{n-3} \wedge \omega_{1} \wedge (\nabla_v \omega_{0})^2 + \dots + \omega_{0}^{n-4}\wedge \omega_1^{2} \left( \nabla_v \omega_0 \wedge \nabla_v \omega_1\right )  + \dots + \omega_{0}^{n-5}\wedge \omega_{1}^{3}\wedge (\nabla_v \omega_0)^2 
\nonumber
\\
&+ \dots + \omega_{1}^{n-2} \wedge \left(  \nabla_v \omega_0 \wedge \nabla_v \omega_1  \right) + \dots + \omega_0 \wedge \omega_{1}^{n-3} (\nabla_v \omega_{0})^2 + \dots
\nonumber
\\
                                     &+ \omega_{1}^{n-2} (0) \wedge (\nabla_v \omega_1)^2   +  \dots \big{)}
\nonumber                                     
\end{align}
Factorizing the $(n-2)$-forms which are symplectic on $Z$,  we can rewrite the previous expression as 
\begin{align}
\omega_t^n(s\cdot v) &=s^2 \cdot (\omega_{0}^{n-2} (0) \wedge \left(  (\nabla_v \omega_0)^2  +    \nabla_v \omega_0 \wedge \nabla_v \omega_1     +  (\nabla_v \omega_1)^2   \right) + \dots 
\nonumber
\\
                                    &+  \omega_{0}^{n-k} \wedge \omega_{1}^{k} \wedge \left(  (\nabla_v \omega_0)^2  +    \nabla_v \omega_0 \wedge \nabla_v \omega_1     +  (\nabla_v \omega_1)^2   \right) + \dots 
\nonumber
\\
                                    &+  \omega_{1}^{n-2} (0) \wedge \left(  (\nabla_v \omega_0)^2  +    \nabla_v \omega_0 \wedge \nabla_v \omega_1     +  (\nabla_v \omega_1)^2   \right) )
\label{eq:Taylor2}
\end{align}
%
As shown in section 2, by restricting the terms $(\nabla_v \omega_0)^2$ and $(\nabla_v \omega_1)^2$ to vectors $\partial_{k_{i}}$ in $K_t$ we have  $(\nabla_v \omega_0)^2 (\partial_{k_{1}}, \dots , \partial_{k_{4}}) = D_{K_{0}}^2 \geq 0$ and $(\nabla_v \omega_1)^2 (\partial_{k_{1}}, \dots , \partial_{k_{4}}) = D_{K_{1}}^2 \geq 0$.   Thus, in the equation \eqref{eq:Taylor2}, the square binomial terms are non-negative
$$(  (\nabla_v \omega_0)|_{K}^2  +    \nabla_v \omega_0 \wedge \nabla_v \omega_1 |_{K}    +  (\nabla_v \omega_1)|_{K}^2 )= (\nabla_v \omega_0|_{K} + \nabla_v \omega_1|_{K})^{2}  := \nabla_v \omega_{t}^{2}|_{K} \geq 0 $$
%
%
and the forms $\omega_{0}^{n-k} \wedge \omega_{1}^{k} $, for $k\in \lbrace 0, 1, \dots, n \rbrace$ are positive on the symplectic subspace in $Z$,   from which we conclude that $\omega_{t}^{n} \geq 0$ on the tubular neighbourhood of the singular locus. 
%
\\
\\
\noindent {\bf Step 2: Poincar{\'e} Lemma}
\\
\noindent These next two steps follow the lines of Perutz \cite{Perutz}, where we first use an application of the Poincar{\'e} Lemma. 
The De Rham homotopy operator
\begin{align}
Q\colon \Omega^k &\rightarrow \Omega^{k-1} 
\nonumber
\\
Q\Omega &= \int_{0}^{1} h^{*}_{t} (\iota_R \Omega) dt 
\nonumber
\end{align}
satisfies 
$$ \textnormal{Id}(\Omega)  - (\iota \circ \pi)^{*}(\Omega) = dQ(\Omega) + Qd(\Omega)$$
Here $\pi\colon N_Z \rightarrow Z$ is the bundle projection, $i\colon Z\rightarrow N_Z$ ithe zero-section,  $h_t\colon N_Z \rightarrow N_Z \, , \, x \rightarrow t\cdot x$ the fibrewise dilation, and $R$ the Euler vector field.
Applying this lemma to a neighbourhood of the zero section $\mathcal{U}_0\subset N_Z$ we find a 1-form $\lambda_t:= Q(\omega_t)$ satisfying $d\lambda_t = \omega_t$  on $\mathcal{U}_0 \setminus Z$. Moreover, notice that $\omega_t$ vanishes up to degree 1 on $K_t$. Inserting the Euler vector field $R$ into $\omega_t$ adds one degree more and produces a 1-form $\iota_R \omega_t$ that vanishes on $Z$ up to degree 2. 
%
\\
\\
\noindent {\bf Step 3} 
\\
We proceed in two parts.  First we consider the case on $\mathcal{U}_0 \setminus Z$, where the argument is very similar as in dimension 4.  Then, we focus on the symplectic subspace inside $Z$.  

\noindent On $\mathcal{U}_{0} \setminus Z$, where $\omega_t$ is near-symplectic, introduce vector fields $X_t$  defined by 
\begin{equation}\label{Moser}
 \iota_{X_t} \omega_t + \lambda_t = 0 
\end{equation}
We want to show that, on the tubular neighbourhood, $X_t$ shrinks as it approaches $Z$.  On the other four complementary directions defined by $K_t$, we have that $\nabla \lambda_t (u) = 0$ for all non-zero vectors $u\in N_{Z_{0}}$, since $\lambda_t$ vanishes to the second order along Z. Furthermore, $\omega_t$ degenerates on $K_t$, and a Taylor expansion shows that $\nabla \omega_t \not= 0 $ on $K$, so that $|X^{t}_{\textnormal{K}}(x)| \leq C |x|$ for a constant $C$, as shown in \cite{Perutz}.

On the symplectic subspace in $Z$ we have, $\lambda_t|_{\textnormal{Symp}_Z } = 0$, but the restriction $\omega_t|_{\textnormal{Symp}_Z  }$ is non-degenerate on this subspace.  Thus, in order to satisfy equation \ref{Moser},  the vector field $X_t$ needs to vanish on $\textnormal{Symp}_Z $. In particular, the components of the vector field along the symplectic subspace satisfy  $|X^{t}_{\textnormal{Symp}}(x)| \leq c |x|$  for a constant $c$.

The family $\lbrace X_t \rbrace_{t\in [0,1]}$ generates a flow $\lbrace \psi_t \rbrace_{t\in [0,1]}$ on $\mathcal{U}_0$ outside $Z$.  A trajectory $x_s$ defined on some interval $[0, \tilde{s}]$ satisfies $\frac{d}{ds}(\log|x_s|) \geq -C  $.  Integrating over $[0, \tilde{s}]$ we obtain  $|x_{\tilde{s}}| \geq e^{C\tilde{s}} |x_0|$. This shows that the trajectory stays inside $\mathcal{U}_0 \setminus Z_0$, hence the flow $\psi_s$ is well defined.
%
%
\\
\\
\noindent {\bf Step 4}
\\
Define on $\mathcal{U}_0 \setminus Z_0$
$$\tilde{\omega}_t:= \psi_t^{*} \omega_t$$ 
and for $p\in Z_0$
$$ \tilde{\omega}_t:=\omega_{Z_{0}}$$
Moser's argument shows that $ \tilde{\omega}_t = \omega_t$ in some neighbourhood of $Z$.  The diffeomorophism $\psi_1$ is not defined on $Z$.  Extend it to $Z$ by the identity.  At the level of the singular locus, we can take the diffeomorphism to be the one from the theorem's assumption $Z_0 \approx Z_1$. This leads to a homeomorphism, which is a diffeomorphism away from $Z$. 
Finally, set $\varphi = \psi_1$, and $\psi_1(\mathcal{U}_0)= \mathcal{U}_1 $. Then we have that $\varphi^{*}\omega_1 = \omega_0$ away from $Z$, but by assumption $\omega_1$ and $\omega_0$ agree on $Z$. \qed

%
%
\subsection{Proof of Corollary 1}
The following proof follows from the previous theorem and an adaptation of an argument from \cite{Perutz}. 
\begin{proof}
Let $\gamma$ be a closed interval inside the line $\varepsilon=\ker(\omega_Z)$.  Consider $\kappa:= \gamma\times B^3  \subset K:= \varepsilon \oplus N_Z$.  Identify an open subset of $Z$ with $V \times \lbrace 0 \rbrace \subset U \simeq V \times B^{3}_{0} (R)$ inside $M$, such that $\kappa \subset U$.  Denote by $z_0$ the coordinate on $\gamma$ and by $\partial_{z_{0}}$  a positively oriented vector field on $\gamma$ for the orientation determined by $\omega$.

Take a metric $g$ for which $\omega|_{\kappa}$ is self-dual.  We can find an orthonormal frame $(e_1, e_2, e_3)$ for $N_Z$ such that  $L_{-} = \textnormal{span}(e_1), L_{+} = \textnormal{span}(e_2, e_3)$.  The metric and the choice of $e_i$ provide normal coordinates $(\bar{z}, z_0, x_1, x_2, x_3)$ on a small neighbourhood of $p$  inside $U$, where $\bar{z}$ correspond to the $(2n-4)$-complementary coordinates to $z_0$ on $Z$.  Using these coordinates we can write three basis elements $\beta_i$ of $\Lambda^{2}_{+}K^{*}$.
\begin{align}
\beta_1 &= dz_0 \wedge dx_1 + dx_2 \wedge dx_3 
\nonumber
\\
\beta_2 &= dz_0 \wedge dx_2 - dx_1 \wedge dx_3 
\nonumber
\\
\beta_3 &= dz_0 \wedge dx_3 + dx_1 \wedge dx_2  
\nonumber
\end{align}
Let $\tilde{F} = \tilde{F}_{-} \oplus \tilde{F}_{+}$ be a matrix representing the linear map $F \in \textnormal{End}(N_Z)$ with respect to $(e_1, e_2, e_3)$, and let $x= (x_1, x_2, x_3)$ and $\beta = (\beta_1, \beta_2, \beta_3)^{T}$. Expand $\omega$ near $Z$ to obtain 
\begin{align}
\omega(z, x) &= \omega|_Z + \left( x\cdot \tilde{F} \cdot \beta^T + O (x^2) \right)
\nonumber
\\
&= \omega|_Z + \left( x_1 \tilde{F}_{-} \beta_1 +  (x_2, x_3) \tilde{F}_{+} {\beta_2\choose \beta_3} + O (x^2) \right)
\nonumber
\end{align}
Define on a small neighbourhood of $Z$ a family of near-symplectic forms with common singular locus $Z$ by 
$$\omega_{t} = (1-t) \omega + t\cdot (\omega|_{Z} -2 x_1 \beta_1 + x_2 \beta_2 + x_3 \beta_3 )$$
Following the same reasoning as in the proof of the previous theorem in a local setting, we can show that this is a family of near-symplectic forms with common degeneracy locus $Z$.  The next steps follow as in the previous proof. 
\end{proof}

\bibliographystyle{alpha}
\bibliography{references}
\end{document}